\documentclass[11pt,reqno]{amsart}
\usepackage[foot]{amsaddr}

\usepackage{fullpage}
\usepackage{amssymb,amsfonts,amsmath,amsthm}

	\theoremstyle{definition}
	\newtheorem{thm}{Theorem}
	\newtheorem*{thm*}{Theorem}

	\newtheorem{conj}[thm]{Conjecture}
	\newtheorem{quest}[thm]{Question}

\usepackage{enumerate}
\usepackage[shortlabels]{enumitem}
\newlist{sumside}{itemize}{2}
\setlist[sumside]{topsep=0pt,label=-,leftmargin=2em,noitemsep}

\usepackage{hyperref}
\usepackage[dvipsnames]{xcolor}
\newcommand\myshade{85}

\definecolor{darkolivegreen}{rgb}{0.33, 0.42, 0.18}
\definecolor{darklavender}{rgb}{0.45, 0.31, 0.59}
\definecolor{darkraspberry}{rgb}{0.53, 0.15, 0.34}
\definecolor{darkred}{rgb}{0.55, 0.0, 0.0}

\colorlet{mylinkcolor}{darkred}
\colorlet{mycitecolor}{YellowOrange}
\colorlet{myurlcolor}{Aquamarine}
\hypersetup{
	linkcolor  = mylinkcolor,%!\myshade!black,
	citecolor  = mycitecolor!\myshade!black,
	urlcolor   = myurlcolor!\myshade!black,
	colorlinks = true,
}

\usepackage{tikz}

\allowdisplaybreaks
\newcommand{\ignore}[1]{}
\usepackage{parskip}

\usepackage{array}

\usepackage[largesc,theoremfont]{newtxtext}
\usepackage[libertine,cmbraces,varbb]{newtxmath}
\DeclareSymbolFont{largesymbols}{LMX}{zplm}{m}{n}

\usepackage{etoolbox}
\patchcmd{\section}{\scshape}{\bfseries}{}{}
\makeatletter
\renewcommand{\@secnumfont}{\bfseries}
\makeatother

\makeatletter
\patchcmd{\@settitle}{\uppercasenonmath\@title}{\Large}{}{}
\patchcmd{\@setauthors}{\MakeUppercase}{\normalsize}{}{}
\makeatother

\makeatletter
\def\th@definition{%
	\thm@headfont{\bfseries}
	\normalfont 
}
\makeatother

\DeclareMathOperator{\dilog}{L}
\DeclareMathOperator{\Li}{Li}
\DeclareMathOperator{\Bern}{B}

\begin{document}
\thispagestyle{empty}

\title{Searching for modular companions}

\author{Shashank Kanade\textsuperscript{1}}
\address{\textsuperscript{1} University of Denver, Denver, USA}
\thanks{I am presently supported
by a start-up grant provided by University of Denver.
I am extremely grateful to Matthew C.\ Russell for his many collaborations
and for suggesting improvements to the present article.
He also graciously provided the proof that the (new) series expression \eqref{eqn:cap1alt} for the first Capparelli identity equals the one found in \cite{KR2}. Thanks are also in order to Jehanne Dousse and Min-Joo Jang for illuminating discussions on radial limits.
}
\email{shashank.kanade@du.edu}

\begin{abstract}
In this note, we report on the results of a computer search performed to find possible modular companions to certain $q$-series identities and conjectures. For the search, we use conditions arising from the  asymptotics of Nahm sums. We focus on two sets of identities: Capparelli's identities, and certain partition conjectures made by the author jointly with Matthew C.\ Russell. 
\end{abstract}

\maketitle

\section{Motivation}

Recently, in several joint works with Matthew C.\ Russell and Debajyoti Nandi \cite{KR1,KR2,KNR}, we have undertaken various searches for discovering integer partition identities of Rogers-Ramanujan-type. 
In \cite{KR1} and \cite{KNR}, we generated many ``difference conditions'' analogous to the ones in Rogers-Ramanujan-type identities and then used Euler's algorithm (which factorizes a formal power series into an infinite product, see \cite{A}) to identify those conditions whose corresponding generating functions may equal periodic infinite products.
A search in the philosophically reverse direction was performed in \cite{KR2}. 
In \cite{KR2}, we started with principally specialized characters of certain level $2$ standard modules for the affine Lie algebra $A_9^{(2)}$
(such characters are always automatically periodic infinite products) and we found  ``analytic sum-sides'' and difference conditions corresponding to these products.
These identities have been recently proved by Bringmann, Jennings-Shaffer and Mahlburg \cite{BJM}. A few identities from \cite{KR1, KR2, Rth} (unrelated to $A_9^{(2)}$) remain unproven.

In all of the searches mentioned above, infinite products were involved in some way or the other. Among these, symmetric products correspond to $q$-series that are essentially modular. Expanding the search to include $q$-series that are modular but not-necessarily infinite products is a natural next step.
In this  note, we focus on two sets of identities/conjectures: Capparelli's identities  (which arise out of level $3$ standard modules for $A_2^{(2)}$) and some conjectures made by the author jointly with Matthew C.\ Russell \cite{KR1} which have product sides that are periodic modulo $9$. 
The reason to consider these two sets of identities is that their analytic sum-sides look very similar.  
We investigate, essentially employing a theorem of Vlasenko and Zwegers \cite{VZ} on asymptotics of certain types of $q$-series, whether perturbing the linear term in the exponent of $q$ in the summands of the analytic sum-sides (as given by \cite{KR2,Kur1} in the first case and \cite{Kur2} in the second) for these identities may lead to further modular companions. 
We do not find any new companions.

A Maple program was used for the explorations in this note. This program and its \texttt{PDF} printout
are attached as ancillary files with this article on \texttt{arXiv}. 

\bigskip

%at:\\
%\url{https://cs.du.edu/~shakanad/assets/ModularCompanionsProgram.mw}.\\
%Its output is available at:\\
%\url{https://cs.du.edu/~shakanad/assets/ModularCompanionsProgram.pdf}.

\section{Asymptotics and necessary conditions for modularity}

%\section{Asymptotics and the conditions for modularity}

We define the following basic Pochhammer symbols:
\begin{align}
(z;q)_j &= \prod_{0\leq t < j}(1-zq^t),\qquad 
(z;q)_\infty =\prod_{0\leq t }(1-zq^t),\qquad 
(z_1,z_2,\dots,z_k;\,q)_\infty =\prod_{1\leq t \leq k}(z_t;q)_{\infty}.
\end{align}

Define the polylogarithm function 
\begin{align}
\Li_{m}(z)=\sum_{n\geq 1}\dfrac{z^n}{n^m},
\end{align}
and define the Rogers' dilogarithm function by:
\begin{align}
\dilog(z) = \Li_2(z) + \dfrac{1}{2}\ln(z)\ln(1-z).
\end{align}
It can be easily checked that $\lim\limits_{z\rightarrow 0^+}\dilog(z) = 0$ and $\lim\limits_{z\rightarrow 1^-}\dilog(z) = \frac{\pi^2}{6}$,
and thus we also define:
\begin{align}
\dilog(0)=0,\qquad \dilog(1)=\frac{\pi^2}{6}.
\end{align}

Let $k$ be a positive integer, $A=[A_{ij}]$ be a symmetric positive definite $k\times k$ matrix, $B=[B_i]$ and $J=[J_i]$ be $k\times 1$ vectors such that $J_i$ are positive integers and let $C$ be a constant. 
Letting $n=[n_i]$ a $k\times 1$ vector of non-negative integers, consider the following Nahm-type $q$-series:	
\begin{align}
F_{A,B,C,J}(q)=\sum_{n\geq 0}\dfrac{q^{\frac{1}{2}n^TAn + n^TB+C}  }{(q^{J_1};q^{J_1})_{n_1}(q^{J_2};q^{J_2})_{n_2}\dots (q^{J_k};q^{J_k})_{n_k} }.
\end{align}

Suppose that the system
\begin{align}
1-Q_i^{J_i}=Q_1^{A_{1i}}Q_2^{A_{2i}}\cdots Q_k^{A_{ki}}, \quad 1\leq i\leq k
\end{align}
has a unique solution for $Q_i\in(0,1)$, $1\leq i \leq k$.

Define the following constants:
\begin{align}
\xi_i&=J_i\dfrac{Q_i^{J_i}}{1-Q_i^{J_i}},\\
\alpha &=\sum_{1\leq i\leq k}\dfrac{1}{J_i}\left(\dilog\left(1\right) - \dilog\left(Q_i^{J_i}\right)\right),\\
\widetilde{A}&=A + \mathrm{diag}\{\xi_1,\xi_2,\dots,\xi_k\},\\
\beta&=(\det \widetilde{A})^{-1/2}\prod_{1\leq i\leq k}Q_i^{B_i}\left(1-Q_i^{J_i}\right)^{-1/2},\\
\gamma&=C+\dfrac{1}{24}\sum_{1\leq i\leq k}J_i\dfrac{1+Q_i^{J_i}}{1-Q_i^{J_i}}
\end{align}

Now, for each $1\leq i\leq k$, define a sequence of polynomials $D_p^{(i)}$ ($p\geq 1$) by considering the expansion coefficients of the following series:
\begin{align}
\exp\left(
\left(B_i+\dfrac{\xi_i}{2}\right)t_i\sqrt{\varepsilon}-\sum_{p\geq 3}
\dfrac{J_i^{p-1}}{p!}\Li_{2-p}\left(\dfrac{\xi_i}{J_i+\xi_i}\right)\Bern_p
\left(\dfrac{t_i}{\sqrt{\varepsilon}}\right)\varepsilon^{p-1}\right) = 1 + \sum_{p=1}^{\infty}D_p^{(i)}(t_i)\varepsilon^{p/2}.
\end{align}
Note that $\frac{\xi_i}{(J_i+\xi_i)}=Q_i^{J_i}$.
Here, $\Bern_p$ denotes the $p$th Bernoulli polynomial.
Also, for any fixed $i$, $D_p^{(i)}$ are polynomials in the variable $t_i$. 
$D_p^{(i)}$ depends on $B_i$ and also on the matrix $A$ via $\xi_i$. We will suppress this dependence.

Define the polynomials  $C_p(t_1,t_2,\dots,t_k)$ ($p\geq 1$) by:
\begin{align}
C_p(t_1,t_2,\dots,t_k)=\sum_{\substack{p_1+p_2+\cdots+p_k=p\\ p_1,p_2,\dots,p_k\geq 0}}D^{(1)}_{p_1}(t_1)
D^{(2)}_{p_2}(t_2)\dots D^{(k)}_{p_k}(t_k).
\end{align}

Finally define the constants $c_p$, $p\geq 1$ by:
\begin{align}
c_p=\dfrac{(\det \widetilde{A})^{1/2}}{(2\pi)^{r/2}}\int C_{2p}(t_1,t_2,\dots,t_k)e^{-\frac{1}{2}t^T\widetilde{A}t}dt_1dt_2\cdots dt_k,
\end{align}
where $t=[t_1,t_2,\dots,t_k]^T$.

We have the following theorem, which is a mild generalization of the theorem of Vlasenko and Zwegers that pertains to the case $J_1 = J_2 = \dots =J_k=1$.
\begin{thm}[cf.\ \cite{VZ}]
Let $q=e^{-\varepsilon}$ with $\varepsilon>0$.
As $\varepsilon\rightarrow 0^+$, we have the following asymptotic expansion:
\begin{align}
F_{A,B,C,J}(q)\sim\beta \cdot e^{\alpha/\varepsilon}\cdot e^{-\gamma\varepsilon}\cdot \left(1 + \sum_{p=1}^{\infty}c_p\varepsilon^p \right).
\end{align}
\end{thm}
\begin{proof}
	The proof is just a careful reworking of the proof of Vlasenko and Zwegers.
\end{proof}

We now recall the modular constraints from Vlasenko and Zwegers, which they proved in the case $J_1 = J_2 = \dots =J_k=1$, which are easily generalizable to suit our case.
\begin{thm}[cf.\ \cite{VZ}]
Suppose $F_{A,B,C,J}(q)$ is a modular form of weight $w$ for some subgroup $\Gamma \leq SL(2,\mathbb{Z})$ of finite index.
Then,
\begin{enumerate}
	\item The weight $w$ has to be $0$.
	\item $\alpha\in\pi^2\mathbb{Q}$.
	\item $e^{-\gamma\varepsilon}\cdot \left(1 + \sum\limits_{p=1}^{\infty}c_p\varepsilon^p \right)=1$.
	This leads to an infinite list of constraints, first four of which are given by:
	\begin{align}
	0&=c_1 - \gamma,\label{eqn:c1gamma}\\
	0&=c_2 - c_1\gamma+\frac{\gamma^2}{2},\label{eqn:c2gamma}\\
	0&=c_3 - c_2\gamma+c_1\frac{\gamma^2}{2} - \frac{\gamma^3}{3!},\label{eqn:c3gamma}\\	
	0&=c_4 - c_3\gamma+c_2\frac{\gamma^2}{2} - c_1\frac{\gamma^3}{3!} + \frac{\gamma^4}{4!}\label{eqn:c4gamma}.
	\end{align}
	In particular, \eqref{eqn:c1gamma} determines the constant $C$ in terms of the matrix $A$ and the vector $B$.
\end{enumerate}
\end{thm}

\section{Capparelli's identities}

Consider Capparelli's identities \cite{Cap}.
\begin{thm}[\cite{Cap}] Let $n$ be a positive integer.
	\begin{enumerate}
		\item The number of partitions of $n$ in which no part is equal to $1$, difference between adjacent parts is at least $2$, and is at least $4$ unless the adjacent parts add up to a multiple of $6$ is the same as the coefficient of $q^n$ in 
		$1/(q^2,q^3,q^9,q^{10};\, q^{12})_{\infty}$.
		\item The number of partitions of $n$ in which no part is equal to $2$, difference between adjacent parts is at least $2$, and is at least $4$ unless the adjacent parts add up to a multiple of $6$ is the same as the coefficient of $q^n$ in 
		$   (q^2,q^{10};\,q^{12})_{\infty}/(q;\,q^2)_{\infty}$.
	\end{enumerate}
\end{thm}
These identities arise out of the principally specialized characters of level $3$ standard modules for the  affine Lie algebra $A_2^{(2)}$.

Analytic sum-sides for the partitions satisfying the difference conditions were found independently in \cite{KR2} and \cite{Kur1}, using different techniques from each other. With these, the identities could be expressed purely using $q$-series.
\begin{thm}[Capparelli's identities, \cite{KR2, Kur1}] We have the following equalities of $q$-series.
\begin{align}
\label{eqn:cap1}
&\sum\limits_{n_1,n_2\geq 0} \dfrac{q^{2n_1^2+6n_1n_2+6n_2^2}}{(q;q)_{n_1}(q^3;q^3)_{n_2}}
=\dfrac{1}{(q^2,q^3,q^9,q^{10};\, q^{12})_{\infty}}\\
&\qquad\qquad
%\stackrel{?}{=}
=
\sum\limits_{n_1,n_2\geq 0} \dfrac{q^{2n_1^2+6n_1n_2+6n_2^2+n_1 }}{(q;q)_{n_1}(q^3;q^3)_{n_2}}
+
\sum\limits_{n_1,n_2\geq 0} \dfrac{q^{2n_1^2+6n_1n_2+6n_2^2+4n_1+6n_2+2  }}{(q;q)_{n_1}(q^3;q^3)_{n_2}},
\label{eqn:cap1alt}\\
\label{eqn:cap2}
&\sum\limits_{n_1,n_2\geq 0} \dfrac{q^{2n_1^2+6n_1n_2+6n_2^2+n_1 }}{(q;q)_{n_1}(q^3;q^3)_{n_2}}
+
\sum\limits_{n_1,n_2\geq 0} \dfrac{q^{2n_1^2+6n_1n_2+6n_2^2+4n_1+6n_2+1  }}{(q;q)_{n_1}(q^3;q^3)_{n_2}}
\notag\\
&\qquad\qquad =
\sum\limits_{n_1,n_2\geq 0} \dfrac{q^{2n_1^2+6n_1n_2+6n_2^2+n_1 +3 n_2}}{(q;q)_{n_1}(q^3;q^3)_{n_2}}
+
\sum\limits_{n_1,n_2\geq 0} \dfrac{q^{2n_1^2+6n_1n_2+6n_2^2+3n_1+6n_2+1  }}{(q;q)_{n_1}(q^3;q^3)_{n_2}}
=
\dfrac{(q^2,q^{10};\,q^{12})_{\infty}}{(q;\,q^2)_{\infty}}.
\end{align}
In the first identity, the very first expression was found independently in \cite{KR2} and \cite{Kur1}, the third expression appears to be new. 
We thank Matthew C.\ Russell for providing a proof that the third expression is indeed equal to the first. A sketch of the proof is as follows.
Given a series $f(x,q)=\sum_{m,n\geq 0} a_{m,n} x^m q^n$, consider $\widetilde{f}(x,q)=\sum_{m,n\geq 0} a_{m,n} x^m q^{3m(m-1)/2 +  n}$. We say that $\widetilde{f}$ is obtained from $f$ by inserting a $3$-staircase (cf.\ \cite{KR2}).
Also recall the following fundamental $q$-series identities due to Euler:
\begin{align}
\left(x;q\right)_\infty^{-1}&=
\sum_{n\geq 0}\dfrac{x^n}{(q;q)_n},\qquad
\left(-x;q\right)_\infty=
\sum_{n\geq 0}\dfrac{x^nq^{n(n-1)/2}}{(q;q)_n}\label{eqn:euler}.
\end{align}
Consider:
\begin{align}
g(x,q)&=(1+xq^2)\dfrac{(-xq^3;q)_{\infty}}{(x^2q^3;q^3)_\infty} 
=\sum_{n_1,n_2\geq 0}\dfrac{x^{n_1+2n_2}q^{n_1^2/2 + 5n_1/2 + 3n_2}}{(q;q)_{n_1}(q^3;q^3)_{n_2}}
+ \sum_{n_1,n_2\geq 0}\dfrac{x^{n_1+2n_2+1}q^{n_1^2/2 + 5n_1/2 + 3n_2+2}}{(q;q)_{n_1}(q^3;q^3)_{n_2}},
\end{align}
where the second equality follows by \eqref{eqn:euler}.
It can be checked easily that $\widetilde{g}(x,q)\big\vert_{x=1}$ equals the double-sum expression in \eqref{eqn:cap1alt}.
On the other hand, clearly,
\begin{align}
g(x,q)=\dfrac{(-xq^2;q)_\infty}{(x^2q^3;q^3)_\infty}
= \sum_{n_1,n_2\geq 0}\dfrac{x^{n_1+2n_2}q^{n_1^2/2 + 3n_1/2 + 3 n_2}}{(q;q)_{n_1}(q^3;q^3)_{n_2}},
\label{eqn:cap1altproof}
\end{align}
and inserting a $3$-staircase in the right-hand side of \eqref{eqn:cap1altproof} and then letting $x=1$ gets us the first sum in \eqref{eqn:cap1}.

In \eqref{eqn:cap2}, the first expression was found in \cite{KR2} and the second in \cite{Kur1}.
Even though the Capparelli identities are by now quite old, such simple term-by-term positive sum-sides were found only recently.
\end{thm}

Now we investigate if the exponent of $q$ in each of the summands could be perturbed by a linear form to arrive at potentially modular $q$-series companions to Capparelli's identities.

First, we look at \eqref{eqn:cap1}.
Let 
\begin{align}
A=\begin{bmatrix}
4 & 6\\ 6 &12
\end{bmatrix},
\quad J_1=1, \quad J_2=3.
\end{align}
 We investigate whether any $q$-series of the following type has a chance to be modular for some $B_1,B_2$ and $C$:
\begin{align}
F_{A,B,C,J} = \sum\limits_{n_1,n_2\geq 0} \dfrac{q^{\frac{1}{2}n^{T}An+n^TB+C }}{(q;q)_{n_1}(q^3;q^3)_{n_2}}
\label{eqn:cap1gen}
\end{align}

The required data is as follows.
Here, $Q_1,Q_2$ satisfy:
\begin{align}
1-Q_1 = Q_1^4Q_2^6,\quad 1-Q_2^3=Q_1^6Q_2^{12}.
\end{align}
There exists a unique solution with $Q_1,Q_2\in(0,1)$ given by:
\begin{align}
Q_1=\frac{3}{4},\quad Q_2=\frac{2}{3^{2/3}}.
\end{align}
Thus, we have:
\begin{align}
\xi_1=3,\quad \xi_2=24.
\end{align}
We also have:
\begin{align}
\gamma=C+\frac{29}{12}.
\end{align}
The constraint \eqref{eqn:c1gamma} becomes:
\begin{align}
C=-\frac{1}{24}+  \frac{5}{144}B_2-\frac{1}{36}B_2B_1+\frac{1}{24}B_1+\frac{1}{12}B_1^2+\frac{7}{432}B_2^2.
\end{align}
Putting this value of $C$ in constraints \eqref{eqn:c2gamma}--\eqref{eqn:c4gamma} we get the following:
\begin{align}
0&= -{\frac {B_1^{2}{ B_2}}{288}}+{\frac {5{ B_1}B_2^{2}}{576}}-{\frac {B_1^{3}}{144}}-{\frac {113B_2^{3}
	}{31104}}+{\frac {B_1}{1152}}+{\frac {149B_2}{6912}}-{
	\frac {199B_2^{2}}{20736}}-{\frac {7B_1^{2}}{576}}+{
	\frac {49B_2B_1}{1728}},\\
0&={\frac {31B_1^{2}B_2}{10368}}-{\frac {127B_1B_2^{2}}{62208}}-{\frac {7B_1^{3}}{5184}}+{\frac {79{{
				 B_2}}^{3}}{1119744}}-{\frac {173B_2^{3}B_1}{279936}}+
{\frac {31B_1}{6912}}+{\frac {335B_2}{41472}}-{\frac {
		5641B_2^{2}}{746496}}+{\frac {2105B_2^{4}}{6718464}}\nonumber\\
&-{\frac {37B_1^{2}}{20736}}+{\frac {1093B_2B_1
	}{62208}}-{\frac {31B_2^{2}B_1^{2}}{31104}}-{\frac {7
		B_1^{4}}{5184}}+{\frac {19B_2B_1^{3}}{7776}}\\
0&=-{\frac {523B_1^{3}}{663552}}-{\frac {262765B_2^{2}
	}{31850496}}-{\frac {3157B_1^{2}}{884736}}+{\frac {299B_1^{4}}{221184}}+{\frac {11393B_2}{1474560}}+{\frac {161
		B_1}{49152}}+{\frac {5329B_1^{2}B_2}{1327104}}-{
	\frac {19171B_1B_2^{2}}{7962624}}\nonumber\\
&-{\frac {32849
				 B_2^{3}B_1}{11943936}}+{\frac {12769B_2^{6}}{
		1934917632}}+{\frac {B_1^{6}}{41472}}+{\frac {B_1^{5}}{
		9216}}+{\frac {100903B_2^{5}}{1074954240}}+{\frac {71{{ 
				B_1}}^{4}B_2}{165888}}-{\frac {11B_2^{3}B_1^{3}}{
		2239488}}-{\frac {721B_1^{3}B_2^{2}}{497664}}\nonumber\\
&+{\frac {
		1373B_1^{2}B_2^{3}}{995328}}+{\frac {901B_2^{
			4}B_1^{2}}{17915904}}-{\frac {20657B_1B_2^{4}}{
		35831808}}-{\frac {565B_2^{5}B_1}{17915904}}+{\frac {{
			 B_2}B_1^{5}}{41472}}-{\frac {B_2^{2}B_1^{4}
	}{18432}}+{\frac {160699B_2^{4}}{286654464}}+{\frac {55687{
			 B_2}B_1}{2654208}}\nonumber\\
&+{\frac {2161B_2^{2}B_1^{2
	}}{442368}}-{\frac {1465B_2B_1^{3}}{331776}}+{\frac {
		793B_2^{3}}{143327232}}.
\end{align}

These constraints imply that 
\begin{align}
B_1 = B_2 =0.
\end{align}
The conclusion is that the only possible modular $q$-series of the shape \eqref{eqn:cap1gen} is the first Capparelli $q$-series found by \cite{KR2} and \cite{Kur1}.
For this, letting $B_1=B_2=0$ implies that $C=-\frac{1}{24}$.

We have:
\begin{align}
\alpha= \left(\dilog(1)-\dilog\left(\frac{3}{4}\right)  \right)+\frac{1}{3}\left(\dilog(1)-\dilog\left(\frac{8}{9}\right)  \right)
=\dilog\left(\frac{1}{4}\right)+\frac{1}{3}\dilog\left(\frac{1}{9}\right)
\end{align}

The $e^{\alpha/\varepsilon}$-type term from the asymptotic expansion of $q^{-1/24}/(q^2,q^3,q^9,q^{10};\, q^{12})_{\infty}$
is $\exp({\frac{\pi^2}{18\varepsilon}})$. This is easily obtained using asymptotic properties of theta functions and the Dedekind eta function. The general rule is that this term is equal to $\exp(\frac{P\pi^2}{3M\varepsilon})$ where the product includes $P$ pairs of (symmetric) congruences modulo $M$. Here, $M=12, P=2$.

Comparing, we get the following well-known dilogarithm identity:
\begin{align}
\dilog\left(\frac{1}{4}\right)+\frac{1}{3}\dilog\left(\frac{1}{9}\right)=\dfrac{\pi^2}{18}.
\end{align}

Now we turn to the expressions of the sort encountered in \eqref{eqn:cap1alt} and \eqref{eqn:cap2}.
Here, the sum-sides involve two summands, and thus
we perform a computer search for (positive integral) values of $B_1,B_2,B_1',B_2',C'$
to see if a sum of the form
\begin{align}
q^C\left(\sum\limits_{n_1,n_2\geq 0} \dfrac{q^{2n_1^2+6n_1n_2+6n_2^2+  B_1n_1 + B_2 n_2}}{(q;q)_{n_1}(q^3;q^3)_{n_2}}
+
\sum\limits_{n_1,n_2\geq 0} \dfrac{q^{2n_1^2+6n_1n_2+6n_2^2+  B_1'n_1+B_2'n_2+C' }}{(q;q)_{n_1}(q^3;q^3)_{n_2}}\right)
\label{eqn:cap2gen}
\end{align}
could possibly be modular.
Suppose that the first sum is asymptotically equal ($q=e^{-\varepsilon}, \varepsilon\rightarrow 0^+$) to
\begin{align}
\beta_1 \cdot e^{\alpha/\varepsilon}\cdot e^{-\gamma_1\varepsilon}\cdot \left(1 + \sum_{p=1}^{\infty}c_{p,1}\varepsilon^p \right)
\end{align}
and the second one is equal to
\begin{align}
\beta_2 \cdot e^{\alpha/\varepsilon}\cdot e^{-\gamma_2\varepsilon}\cdot \left(1 + \sum_{p=1}^{\infty}c_{p,2}\varepsilon^p \right).
\end{align}

Therefore, combined, we have the expansion:
\begin{align}
e^{\alpha/\varepsilon}\cdot \left(  \beta_1\cdot e^{-\gamma_1\varepsilon}\cdot \left(1 + \sum_{p=1}^{\infty}c_{p,1}\varepsilon^p \right)
+ \beta_2\cdot e^{-\gamma_2\varepsilon}\cdot \left(1 + \sum_{p=1}^{\infty}c_{p,2}\varepsilon^p \right)  \right).
\end{align}

Now we check if this combined asymptotic can be written in the form
$\lambda\cdot e^{\alpha/\varepsilon}$ for some constants $\lambda,\alpha$.

Within the  range $0\leq B_1,B_2,B_1',B_2',C'\leq 6$, the parameters that seem to satisfy the modular constraints are precisely the ones that correspond to the ones appearing in \eqref{eqn:cap1alt} and \eqref{eqn:cap2}.
\begin{quest}
	Extend the search to more number of sums generalizing the addition of two sums in \eqref{eqn:cap2gen}.
\end{quest}	

\begin{quest}
Prove it rigorously that indeed no other values than the ones given above
in \eqref{eqn:cap1alt} and \eqref{eqn:cap2} 
lead to possibly modular $q$-series of the shape \eqref{eqn:cap2gen}.
\end{quest}

\section{Certain Mod-9 conjectures from \cite{KR1}}

In \cite{KR1}, jointly with Russell, we conjectured the first four of the following identities. Thereafter, Russell provided the fifth companion in \cite{Rth}.
\begin{conj}[\cite{KR1,Rth}] \label{conj:KR}
We say that a partition $n=\lambda_1+\lambda_2+\cdots+\lambda_k$ (written in a weakly increasing order) 
satisfies ``Condition($i$)'' if $\lambda_{j+2}-\lambda_{j}\geq 3$ for all $1\leq j\leq k-2$ and 
$\lambda_{j+1}-\lambda_{j}\leq 1$ implies $\lambda_{j+1}+\lambda_j\equiv i\,\,(\mathrm{mod}\,\,3)$ for all $1\leq j\leq k-1$.
	
We have the following conjectures. 
For all $n\in \mathbb{N}$, the following identities hold.
\begin{enumerate}
\item Partitions of $n$ satisfying Condition($0$) are equinumerous with partitions on $n$ where each part is congruent to either
$\pm 1$, $\pm 3$ modulo $9$.
\item Partitions of $n$ satisfying Condition($0$) with smallest part at least $2$ 
are equinumerous with partitions on $n$ where each part is congruent to either
$\pm 2$, {$\pm 3$} modulo $9$.
\item Partitions of $n$ satisfying Condition($0$) with smallest part at least $3$ 
are equinumerous with partitions on $n$ where each part is congruent to either
$\pm 3$, {$\pm 4$} modulo $9$.
\item Partitions of $n$ satisfying Condition($2$) with smallest part at least $2$ 
are equinumerous with partitions on $n$ where each part is congruent to either
$2,3,5$ or $8$  modulo $9$.
\item 	Partitions of $n$ satisfying Condition($1$) without an appearance of $2+2$ as a sub-partition
are equinumerous with partitions on $n$ where each part is congruent to either
$1,4,6$ or $7$  modulo $9$.
\end{enumerate}	

\end{conj}

As is clear, the products in the first three conjectures are symmetric modulo 9, and can be expressed in terms of  Dedekind eta and theta functions. They are therefore modular up to an appropriate factor of $q^C$.
These products are related to the principally specialized characters of level $3$ standard modules for the affine Lie algebra $D_4^{(3)}$.

In a fascinating development, Kur\c{s}ung\"oz \cite{Kur2} found analytic sum-sides as generating functions of the partitions satisfying the difference conditions.
However, as yet, these conjectures remain unproved.	
Thus, equivalently, we have the conjectures in the following purely $q$-series form.
\begin{conj}[\cite{Kur2}] We have the following equalities of $q$-series:%Identities in Conjecture \ref{conj:KR} are equivalent to following $q$-series identities:
	\begin{align}
	\sum\limits_{n_1,n_2\geq 0} \dfrac{q^{n_1^2+3n_1n_2+3n_2^2}}{(q;q)_{n_1}(q^3;q^3)_{n_2}}&=\dfrac{1}{(q,q^3,q^6,q^9;q^9)_\infty},\\
	\sum\limits_{n_1,n_2\geq 0} \dfrac{q^{n_1^2+3n_1n_2+3n_2^2+n_1+3n_2}}{(q;q)_{n_1}(q^3;q^3)_{n_2}}&=\dfrac{1}{(q^2,q^3,q^6,q^7;q^9)_\infty},\\
	\sum\limits_{n_1,n_2\geq 0} \dfrac{q^{n_1^2+3n_1n_2+3n_2^2+2n_1+3n_2}}{(q;q)_{n_1}(q^3;q^3)_{n_2}}&=\dfrac{1}{(q^3,q^4,q^5,q^6;q^9)_\infty},\\
	\sum\limits_{n_1,n_2\geq 0} \dfrac{q^{n_1^2+3n_1n_2+3n_2^2+n_1+2n_2}}{(q;q)_{n_1}(q^3;q^3)_{n_2}}&=\dfrac{1}{(q^2,q^3,q^5,q^8;q^9)_\infty},\\
	\sum\limits_{\substack{n_1\geq 0\\n_2\geq 1}} 
	\dfrac{q^{n_1^2+3n_1n_2+3n_2^2 + n_1 + 4n_2}}{(q;q)_{n_1}(q^3;q^3)_{n_2}}
	+\sum\limits_{\substack{n_1\geq 0\\n_2\geq 1}} \dfrac{q^{n_1^2+3n_1n_2+3n_2^2 + 2n_1+4n_2+1}}{(q;q)_{n_1}(q^3;q^3)_{n_2}}
	&+\sum\limits_{\substack{n_1\geq 0}} \dfrac{q^{n_1^2}}{(q;q)_{n_1}}
	=\dfrac{1}{(q^1,q^4,q^6,q^7;q^9)_\infty}.
	\end{align}

\end{conj}

We now vary the term in the exponent to see if there are further $q$-series of a similar shape that may be modular.
Here we have
\begin{align}
A=\begin{bmatrix}
2 & 3\\ 3 & 6
\end{bmatrix},
\quad J_1=1, \quad J_2=3.
\end{align}
For this matrix $A$ the constants $Q_1,Q_2$ satisfy:
\begin{align}
1-Q_1 = Q_1^2Q_2^3,\quad 1-Q_2^3=Q_1^3Q_2^{6}.
\end{align}
It can be checked that the following values satisfy these equations:
\begin{align}
Q_1=1-2\sin\left(\pi/18\right),\quad Q_2^3=4\sin^2\left(\pi/18\right) + 4\sin\left(\pi/18\right),
\end{align}
and it can be checked numerically that indeed there is a unique solution with both $Q_1,Q_2\in (0,1)$.
We have the following minimal polynomials:
\begin{align}
(Q_1)^3 - 3 (Q_1)^2 + 1 = 0, \label{eqn:Q1}\\
(Q_2)^9 - 6(Q_2)^6 + 3(Q_2)^3+1=0.
\end{align}
It can be checked that the $\xi_i$ satisfy the following polynomials:
\begin{align}
(\xi_1)^3 - 3 (\xi_1)^2 - 1 = 0, \\
(\xi_2)^3 - 9(\xi_2)^2 - 54(\xi_2)-27=0.
\end{align}

In this case, the polynomials satisfied by $B_1$ and $B_2$ in order to fulfill the necessary conditions for modularity are quite nasty
and thus we do not reproduce them here.
Numerically, we have found that between the range $-40\leq B_1,B_2\leq 40$, the only values that satisfy these polynomials are
the ones corresponding to the first three conjectures.
\begin{quest}
	It will be interesting to prove that indeed no other values of $B_1$ and $B_2$ possibly lead to modular $q$-series. In particular, it will be interesting to first extend the search to include fractional values of $B_1$ and $B_2$.
\end{quest}
For these three cases, comparing the $e^{\alpha/\varepsilon}$-type term in the asymptotic expansion of the product (here, we have $2$ symmetric pairs of congruences modulo $9$ included, giving us $\exp(\frac{2\pi^2}{27\varepsilon})$), we arrive at the following dilogarithm identity which must be true if the conjectures are to hold:
\begin{align}
\left(\dilog(1)-\dilog\left(Q_1\right)  \right)+\frac{1}{3}\left(\dilog(1)-\dilog\left({Q_2}^3\right)  \right) = \frac{2\pi^2}{27},
\end{align}
or equivalently,
\begin{align}
\dilog\left(Q_1\right)+\frac{1}{3}\dilog\left({Q_2}^3\right) = \frac{4\pi^2}{27}\label{eq:NewDilog}.
\end{align}
This identity can be verified numerically with a high precision using a computer.

\begin{quest}
	In \cite{Lox}, Loxton conjectured several dilogarithm identities involving the roots of the polynomial \eqref{eqn:Q1},
	and also proved two of these identities by first proving new $q$-series identities and then performing an asymptotic analysis using saddle point method.
	It will be interesting to prove and connect \eqref{eq:NewDilog} to the dilogarithm identities in \cite{Lox}.
\end{quest}


\begin{thebibliography}{XX}
\bibitem{A}
G.\ E.\ Andrews,
\emph{q-Series: Their Development and Application in Analysis, 
Number Theory, Combinatorics, Physics and Computer Algebra},
CBMS Regional Conference Series in Math., No.\ 66, 
American Math.\ Soc.\ Providence (1986).

\bibitem{BJM} 
K.\ Bringmann, C.\ Jennings-Shaffer and K.\ Mahlburg,
Proofs and Reductions of Kanade and Russell's partition identities,
\href{https://arXiv.org/abs/1809.06089}{\texttt{arXiv:1809.06089 [math.NT]}}.

\bibitem{Cap}
S.\ Capparelli, 
On some representations of twisted affine Lie algebras and combinatorial identities,
\emph{J.\ Algebra} \textbf{154} (1993), no.\ 2, 335--355.

\bibitem{KR1}
S.\ Kanade and M.\ C.\ Russell,
\texttt{IdentityFinder} and some new identities of Rogers-Ramanujan type,
\emph{Exp. Math.} \textbf{24} (2015), no.\ 4, 419--423.

\bibitem{KR2}
S.\ Kanade and M.\ C.\ Russell,
Staircases to analytic sum-sides for many new integer partition identities of Rogers-Ramanujan type,
\emph{Electron. J. Combin.} \textbf{26} (2019), no.\ 1, Paper 1.6.

\bibitem{KNR}
S.\ Kanade, D.\ Nandi and M.\ C.\ Russell,
A variant of \texttt{IdentityFinder} and some new identities of Rogers-Ramanujan-MacMahon type,
\href{https://arXiv.org/abs/1902.00790}{\texttt{arXiv:1902.00790 [math.CO]}}.

\bibitem{Kur1}
K.\ Kur\c{s}ung\"oz,
Andrews-Gordon Type Series for Capparelli's and G\"ollnitz-Gordon Identities,
\href{https://arXiv.org/abs/1807.11189}{\texttt{arXiv:1807.11189 [math.CO]}}.

\bibitem{Kur2}
K.\ Kur\c{s}ung\"oz,
Andrews-Gordon Type Series for Kanade-Russell Conjectures,
\href{https://arXiv.org/abs/arXiv:1808.01432}{\texttt{arXiv:1808.01432 [math.CO]}}.

\bibitem{Lox} 
J.\ H.\ Loxton,
Special values of the dilogarithm function,
\emph{Acta Arith.} \textbf{43} (1984), no.\ 2, 155--166. 	

\bibitem{Rth}
M.\ C.\ Russell,
\emph{Using experimental mathematics to conjecture and prove theorems in the theory of partitions and commutative and non-commutative recurrences}, Thesis (Ph.D.)--Rutgers The State University of New Jersey - New Brunswick. 2016, ProQuest LLC.

\bibitem{VZ}
M.\ Vlasenko and S.\ Zwegers, 
Nahm's conjecture: asymptotic computations and counterexamples,
\emph{Commun. Number Theory Phys.} \textbf{5} (2011), no.\ 3, 617--642.
	
	
\end{thebibliography}
\end{document}